\documentclass{article}

\usepackage[verbose=true,letterpaper]{geometry}
\AtBeginDocument{
  \newgeometry{
    textheight=9in,
    textwidth=6in,
    top=1in,
    headheight=12pt,
    headsep=25pt,
    footskip=30pt
  }
}
\usepackage[utf8]{inputenc} % allow utf-8 input
\usepackage[T1]{fontenc}    % use 8-bit T1 fonts
\usepackage{lmodern}
\usepackage{hyperref}       % hyperlinks
\usepackage{url}            % simple URL typesetting
\usepackage{booktabs}       % professional-quality tables
\usepackage{amsfonts}       % blackboard math symbols
\usepackage{nicefrac}       % compact symbols for 1/2, etc.
\usepackage{microtype}      % microtypography
\usepackage{xcolor}         % colors
\definecolor{linkcolor}{RGB}{83,83,182}
\definecolor{citecolor}{RGB}{128,0,128}
\hypersetup{
    colorlinks=true,
    citecolor=citecolor,
    linkcolor=linkcolor,
    urlcolor=linkcolor
}
\usepackage{amsmath}
\usepackage{cleveref}
\usepackage{graphicx}
\usepackage[frozencache,cachedir=.]{minted}
\usepackage{multicol}
\usepackage{wrapfig}

\usepackage[noend]{algorithmic}
\usepackage[ruled,noend]{algorithm2e}

\newcommand{\op}{\mathrm{op}}

\newcommand{\RR}{\mathbb{R}}
\newcommand{\NN}{\mathbb{N}}

\newcommand{\ID}{\mathrm{ID}}
\newcommand{\AD}{\mathrm{AD}}
\newcommand{\OS}{\mathrm{OS}}

\newtheorem{lemma}{Lemma}

\newtheorem{corollary}{Corollary}

\newtheorem{assumption}{Assumption}
\newtheorem{remark}{Remark}

\newtheorem{example}{Example}
\newenvironment{proof}[1][]{\noindent {\bf Proof #1:\;}}{\hfill $\Box$}

\title{One-step differentiation of iterative algorithms}

\author{%
  Jérôme Bolte\footnote{%
  Toulouse School of Economics,
  University of Toulouse Capitole.
  Toulouse, France.}
\and
  Edouard Pauwels\footnote{%
  IRIT, CNRS, Université de Toulouse.
  Institut Universitaire de France (IUF).
  Toulouse, France.}
\and
  Samuel Vaiter\footnote{%
  CNRS \& Université Côte d'Azur,
  Laboratoire J. A. Dieudonné.
  Nice, France.}
}

\begin{document}

\maketitle

\begin{abstract}
    In appropriate frameworks, automatic differentiation is transparent to the user at the cost of being a significant computational burden when the number of operations is large.
    For iterative algorithms, implicit differentiation alleviates this issue but requires custom implementation of Jacobian evaluation.
    In this paper, we study one-step differentiation, also known as Jacobian-free backpropagation, a method as easy as automatic differentiation and as performant as implicit differentiation for fast algorithms (e.g., superlinear optimization methods).
    We provide a complete theoretical approximation analysis with specific examples (Newton's method, gradient descent) along with its consequences in bilevel optimization. Several numerical examples illustrate the well-foundness of the one-step estimator.
\end{abstract}

\section{Introduction}

Differentiating the solution of a machine learning problem is a important task, e.g., in hyperparameters optimization~\cite{bengio2000gradient}, in neural architecture search~\cite{liu2019darts} and when using convex layers~\cite{agrawal2019differentiable}.
There are two main ways to achieve this goal: \emph{automatic differentiation} (AD) and \emph{implicit differentiation} (ID).
Automatic differentiation implements the idea of evaluating derivatives through the compositional rules of differential calculus in a user-transparent way. It is a mature concept~\cite{griewank2008evaluating} implemented in  several machine learning frameworks \cite{pytorch2019,jax2018github,tensorflow2015whitepaper}. However, the time and memory complexity incurred may become prohibitive as soon as the computational graph becomes bigger, a typical example being unrolling iterative optimization algorithms such as gradient descent \cite{bai2019deep}.
The alternative, implicit differentiation, is not always accessible: it does not solely relies on the compositional rules of differential calculus and usually requires solving a linear system. The user needs to implement custom rules in an automatic differentiation framework (as done, for example, in \cite{amos2017optnet}) or use dedicated libraries such as~\cite{jaxopt,agrawal2019differentiable,bertrand2020implicit} implementing these rules for given models.
Provided that the implementation is carefully done, this is most of the time the gold standard for the task of differentiating problem solutions.

\paragraph{Contributions.}
We study a \emph{one-step Jacobian} approximator based on a simple principle: differentiate only the last iteration of an iterative process and drop previous derivatives. The idea of dropping derivatives or single-step differentiation was explored, for example, in \cite{geng2021attention,ramzi2022shine,WuFung2022JFB,sahoo2023backpropagation} and our main contribution is a general account and approximation analysis for Jacobians of iterative algorithms.
One-step estimation constitutes a rough approximation at first sight and our motivation to study is its ease-of-use within automatic differentiation frameworks: no custom Jacobian-vector products or vector-Jacobian products needed as long as a \texttt{stop\_gradient} primitive is available (see \Cref{tab:presentationThreeAlgo}).

We conduct a thorough approximation analysis of one-step Jacobian (\Cref{cor:approxOS}). The distance with the true Jacobian are produced by the distance to the solution (i.e., the quality of completion of the optimization phase) and the lack of contractivity of the iteration mapping. These imprecisions have  to be balanced with the ease of implementation of the one-step Jacobian. This suggests that one-step differentiation is efficient for small contraction factors, which corresponds to fast algorithms. We indeed show that one-step Jacobian is asymptotically correct for super-linearly convergent algorithms (\Cref{cor:superlinearCvgt}) and provide similar approximation rate as implicit differentiation for \emph{quadratically convergent algorithms} (\Cref{cor:quadCvgt}).

We exemplify these results with hypergradients in bilevel optimization, conduct a detailed complexity analysis and highlight in \Cref{cor:bilevel} the estimation of approximate critical points.
Finally, numerical illustrations are provided to show the practicality of the method on logistic regression using Newton's algorithm, interior point solver for quadratic programming and weighted ridge regression using gradient descent.

\paragraph{Related works.}
Automatic differentiation~\cite{griewank2008evaluating} was first proposed in the forward mode in~\cite{wengert1964simpleautomatic} and its reverse mode in~\cite{linnainmaa1970representationcumulative}.
The study of the behaviour of differentiating iterative procedure by automatic differentiation was first analyzed in~\cite{gilbert1992automatic} and~\cite{beck1994automatic} in the optimization community.
It was studied in the machine learning community for smooth methods~\cite{mehmood2020automatic,lorraine2020optimizing,pauwels2022derivatives}, and nonsmooth methods~\cite{bolte2022automatic}.
Implicit differentiation is a recent highlight in machine learning.
It was shown to be a good way to estimate the Jacobian of problem solutions, for deep equilibrium network~\cite{bai2019deep}, optimal transport~\cite{luise2018differential}, and also for nonsmooth problems~\cite{bolte2021nonsmooth}, such as sparse models~\cite{bertrand2020implicit}.
Truncated estimation of the "Neumann series" for the implicit differentiation is routinely used~\cite{lorraine2020optimizing,luketina2016scalable}.
In the very specific case of min-min problems, \cite{ablin2020super} studied the speed of convergence of automatic, implicit, and analytical differentiation.

The closest work to ours is~\cite{WuFung2022JFB} -- under the name \emph{Jacobian-free backpropagation} -- but differs significantly in the following ways.
Their focus is on single implicit layer networks, and guarantees are qualitative (descent direction of~\cite[Theorem 3.1]{WuFung2022JFB}). In contrast, we provide quantiative results on abstract fixed points and  applicable to any architecture.
The idea of ``dropping derivatives'' was proposed in ~\cite{finn2017maml} for meta-learning, one-step differentiation was also investigated to train Transformer architectures \cite{geng2021attention}, and to solve bilevel problems with quasi-Newtons methods~\cite{ramzi2022shine}.

\begin{table}[t]
    \centering
    \begin{minipage}[t]{.3\textwidth}
\begin{algorithm}[H]
  \caption{Automatic}
	\label{alg:autodiff}
	\textbf{Input:} \\$\theta \mapsto x_0(\theta) \in \mathcal{X}$, $k > 0$.\\
    \textbf{Eval:}\\
   \begin{algorithmic}
   \STATE \texttt{with\_gradient}
    \FOR{$i= 1, \ldots, k$}
        \STATE $x_{i}(\theta) = F(x_{i-1}(\theta), \theta)$
    \ENDFOR
    \STATE \textbf{Return:} $x_k(\theta)$
%    \STATE \BlankLine
  \end{algorithmic}
  \textbf{Differentiation:} native autodiff on \textbf{Eval}.
  
\end{algorithm}
\end{minipage}
\quad
\begin{minipage}[t]{.3\textwidth}
\begin{algorithm}[H]
  \caption{Implicit}
	\label{alg:implicitDiff}
	\textbf{Input:}\\ $x_0 \in \mathcal{X}$, $k > 0$.\\
    \textbf{Eval: }\\
   \begin{algorithmic}
    \STATE \texttt{stop\_gradient}
    \FOR{$i= 1, \ldots, k$}
        \STATE $x_{i} = F(x_{i-1}, \theta)$
    \ENDFOR
    \STATE \textbf{Return:} $x_k$
  \end{algorithmic}
  \textbf{Differentiation:} 
  Custom implicit VJP / JVP.
\end{algorithm}
\end{minipage}
\quad
\begin{minipage}[t]{.3\textwidth}
\begin{algorithm}[H]
  \caption{One-step}
	\label{alg:oneStep}
	\textbf{Input:} $x_0 \in \mathcal{X}$, $k > 0$.\\
    \textbf{Eval: }\\
    
   \begin{algorithmic}
    \STATE \texttt{stop\_gradient}
    \FOR{$i= 1, \ldots, k-1$}
        \STATE $x_{i} = F(x_{i-1}, \theta)$
    \ENDFOR
    \STATE \texttt{with\_gradient}
    \STATE $x_k(\theta) = F(x_{k-1},\theta)$
    \STATE \textbf{Return:} $x_k(\theta)$
  \end{algorithmic}
  \textbf{Differentiation:} native autodiff on \textbf{Eval}
\end{algorithm}
\end{minipage}
\begin{minipage}[h]{.49\textwidth}
  \caption{Qualitative comparison of differentiation strategies. Native autodiff refers to widespread primitives in differentiable programming (e.g. \texttt{grad} in JAX). Custom JVP/VJP refers to specialized libraries such as \texttt{jaxopt}~\cite{blondel2022efficient}  or \texttt{qpth}~\cite{amos2017optnet} implicit differentiation in specific contexts.}
    \label{tab:presentationThreeAlgo}
\end{minipage}\quad
\begin{minipage}[h]{.48
\textwidth}
\begin{tabular}{lcc}
        \toprule
        Algorithm& Implementation  & Efficient \\ \midrule
         \textbf{Automatic} & Native autodiff  & no \\
         \textbf{Implicit} & Custom& yes \\
         \textbf{One step} & Native autodiff  & yes \\
         \bottomrule
    \end{tabular}
\end{minipage}
\vskip -.2in
\end{table}

\section{One-step differentiation}

\subsection{Automatic, implicit and one-step differentiation}
\label{sec:threeAlgorithms}
Throughout the text $F \colon \RR^n \times \RR^m \to \RR^n$ denotes a recursive algorithmic map from $\RR^n$ to $\RR^n$ with $m$ parameters. For any $\theta \in \RR^m$, we write $F_\theta : x \mapsto F(x,\theta)$ and let $F_\theta^k$ denote $k$ recursive composition of $F_\theta$, for $k \in \NN$.
The map $F_\theta$ defines a recursive algorithm as follows
\begin{equation}\label{eq:algo}
    x_0(\theta) \in \RR^n \quad \text{and} \quad x_{k+1}(\theta) = F(x_k(\theta), \theta) ,
\end{equation}
We denote by $J_x F_\theta$ the Jacobian matrix with respect to the variable $x$.
The following assumption is sufficient to ensure a non degenerate asymptotic behavior.
\begin{assumption}[Contraction]{\rm 
    \label{ass:mainAssumption}
    Let $F \colon \RR^{n} \times \RR^m \to \RR^n$ be $C^1$, $0 \leq \rho<1$, and $\mathcal{X} \subset \RR^n$ be nonempty convex closed, such that for any $\theta \in \RR^m$, $F_\theta(\mathcal{X}) \subset \mathcal{X}$ and $\|J_x F_\theta\|_{\op} \leq \rho$.}
\end{assumption}
\begin{example}
    \label{ex:contractionsOpti}{\rm 
    The main algorithms considered in this paper fall in the scope of smooth optimization. The algorithmic map $F_\theta$ is associated to a smooth parameteric optimization problem given by $f \colon \RR^{n} \times \RR^m \to \RR$ such that $f_\theta \colon x \mapsto f(x,\theta)$ is strongly convex, uniformly in $\theta$. Two examples of algorithmic maps are given by gradient descent, $F_\theta(x) = x - \alpha \nabla f_\theta(x)$, or Newton's $F_\theta(x) = x - \alpha \nabla^2 f_\theta(x)^{-1} \nabla f_\theta(x)$ for positive step $\alpha>0$. For small step sizes, gradient descent provides a contraction and Newton's method provides a local contraction, both fall in the scope of \Cref{ass:mainAssumption}.}
\end{example}

The following lemma gathers known properties regarding the fixed point of $F_\theta$, denoted by $\bar{x}(\theta)$ and for which we will be interested in estimating derivatives. 
\begin{lemma}
    Under \Cref{ass:mainAssumption}, for each $\theta$ in $\RR^m$ there is a unique fixed point of $F_\theta$ in $\mathcal{X}$ denoted by $\bar{x}(\theta)$, which is a $C^1$ function of $\theta$. Furthermore, for all $k \in \NN$, we have $\|x_k(\theta) - \bar{x}(\theta)\| \leq \rho^k \|x_0(\theta) - \bar{x}(\theta)\| \leq  \rho^k \frac{\|x_0 - F_\theta(x_0)\|}{1 - \rho}$.
    \label{lem:fixedPoint}
\end{lemma}
This is well known, we briefly sketch the proof. The mapping $F_\theta$ is a $\rho$ contraction on $\mathcal{X}$ -- use the convexity of $\mathcal{X}$ and the intermediate value theorem. Banach fixed point theorem ensures existence and uniquenness, differentiability is due to the implicit function theorem. The convergence rate is classical. We are interested in the numerical evaluation of the Jacobian $J_\theta \bar x(\theta)$, thus well-defined under \Cref{ass:mainAssumption}.

The automatic differentiation estimator $J^{\AD}x_{k}(\theta) = J_{\theta} x_{k}(\theta)$ propagates the derivatives (either in a forward or reverse mode) through iterations based on the recursion, for $i = 1, \ldots, k-1$,
\begin{equation}\label{eq:prodjac}
    J_\theta x_{i+1}(\theta) = J_x F(x_i(\theta), \theta) J_\theta x_i(\theta) + J_\theta F(x_i(\theta), \theta) .
\end{equation}
Under assumption \ref{ass:mainAssumption} we have $J^{\AD}x_{k}(\theta) \to J_{\theta} \bar{x}(\theta)$ and the convergence is asymptotically linear \cite{gilbert1992automatic,beck1994automatic,mehmood2020automatic,scieur2022curse,bolte2022automatic}. $J^{\AD}x_{k}(\theta)$ is available in differentiable programming framework implementing common primitives such as backpropagation.

The implicit differentiation estimator $J^{\ID}x_{k}(\theta)$ is given by application of the implicit function theorem using $x_k$ as a surrogate for the fixed point $\bar{x}$,
\begin{equation}\label{eq:diffID}
    J^{\ID} x_{k}(\theta) = (I - J_x F(x_k(\theta), \theta))^{-1} J_\theta F(x_k(\theta), \theta) .
\end{equation}
By continuity of the derivatives of $F$, we also have $J^{\ID} x_{k}(\theta) \to J_{\theta} \bar{x}(\theta)$ as $k \to \infty$, (see \textit{e.g.} \cite[Lemma 15.1]{griewank2008evaluating} or \cite[Theorem 1]{blondel2022efficient}). Implementing $J^{\ID}x_{k}(\theta)$ requires either manual implementation or dedicated techniques or libraries \cite{amos2017optnet,agrawal2019differentiable,kolter2020tutorial,gu2020implicit,el2021implicit,blondel2022efficient} as the matrix inversion operation is not directly expressed using common differentiable programming primitives.

The one-step estimator $J^{\OS}x_{k}(\theta)$ is the Jacobian of the fixed point map for the last iteration 
\begin{equation}\label{eq:diffOS}
    J^{\OS}x_{k}(\theta) = J_\theta F(x_{k-1}(\theta), \theta) .
\end{equation}
Contrary to automatic differentiation or implicit differentiation estimates, we do not have $J^{\OS}x_{k}(\theta) \to J_{\theta} \bar{x}(\theta)$ in general as $k \to \infty$, but we will see that  the error is essentially proportional to $\rho$, and thus negligible for fast algorithms for which the estimate is accurate.

From a practical viewpoint, the three estimators $J^\AD$, $J^\ID$ and $J^\OS$ are implemented in a differentiable programming framework, such as \texttt{jax}, thanks to a primitive \texttt{stop\_gradient}, as illustrated by Algorithms 1, 2 and 3.
The computational effect of the \texttt{stop\_gradient} primitive is to replace the actual Jacobian $J_x F(x_i(\theta), \theta)$ by zero for chosen iterations $i \in \{1,\ldots,k\}$.
Using it for all iterations except the last one, allows one to implement $J^\OS$ in~\eqref{eq:diffOS} using Algorithm \ref{alg:oneStep}. This illustrates the main interest of the one-step estimator: it can be implemented using any differentiable programming framework which provides a \texttt{stop\_gradient} primitive and does not require custom implementation of implicit differentiation.
\Cref{fig:newton} illustrates an implementation in \texttt{jax} for gradient descent.
\begin{figure}
    \centering
      \begin{multicols}{2}
\begin{minted}[fontsize=\scriptsize,highlightlines={21-22},highlightcolor=red!20]{python}
def F(x, theta):
    # here a gradient step
    return x - alpha * grad_f(x, theta)

def iterative_procedure(theta, x0):
    x = x0
    for _ in range(100):
        x = F(x, theta)
    return x

# Automatic differentiation
J_ad = jacfwd(iterative_procedure)(theta, x0)

# Implicit differentiation
J_id = # Custom implementation (e.g. jaxopt)

# One-step differentiation
def iterative_procedure_with_stop_grad(theta, x0):
    x = x0
    for _ in range(99):
        x = F(stop_gradient(x),
              stop_gradient(theta))
    x = F(x, theta)
    return x
J_os = jacfwd(iterative_procedure)(theta, x0)
\end{minted}
  \end{multicols}
    \vspace{-0.8cm}
    \caption{Implementation of Algorithms~\ref{alg:autodiff},~\ref{alg:implicitDiff} and \ref{alg:oneStep} in \texttt{jax}. $F$ is a gradient step of some function $f$. The custom implementation of implicit differentiation is not explicited. The function \texttt{stop\_gradient} is present in \texttt{jax.lax} and \texttt{jacfwd} computes the full Jacobian using forward-mode AD.} 
    \label{fig:newton}
\end{figure}

\subsection{Approximation analysis of one step differentiation for linearly convergent algorithms}
The following lemma is elementary.
It describes the main mathematical mechanism at stake behind our analysis of one-step differentiation.
\begin{lemma}\label{lem:matrixIdentity}
    Let $A \in \RR^{n \times n}$ with $\|A\|_{\op} \leq \rho < 1$ and $B,\tilde{B} \in \RR^{n \times m}$, then
    \begin{align*}
      (I -A)^{-1}B - \tilde{B} &= A(I - A)^{-1}  B + B - \tilde{B}.
    \end{align*}
    Moreover, we have the following estimate,
    \begin{align*}
      \|(I -A)^{-1}B - \tilde{B} \|_\op&\leq \frac{\rho}{1 - \rho}\|B\|_\op   + \|B - \tilde{B}\|_\op.
    \end{align*}
\end{lemma}
\begin{proof}
    First for any $v \in \RR^n$, we have $\|(I-A)v\| \geq \|Iv\| - \|Av\| \geq (1-\rho)\|v\|$,
    which shows that $I-A$ is invertible (the kernel of $I-A$ is trivial). We also deduce that $\|(I-A)^{-1}\|_\op \leq 1 / (1 - \rho)$.
    Second, we have $(I-A)^{-1} - I = A(I-A)^{-1}$, since $((I-A)^{-1} - I)(I-A) = A$, and therefore
    \begin{align*}
       A(I - A)^{-1}  B + B - \tilde{B} &=((I-A)^{-1} - I)  B + B - \tilde{B} = (I-A)^{-1}B - \tilde{B}.
    \end{align*}
    The norm bound follows using the submultiplicativity of operator norm, the triangular inequality and the fact that $\|A\|_{\op} \leq \rho$ and $\|(I-A)^{-1}\|_\op \leq 1/(1 - \rho)$.
\end{proof}
\begin{corollary}\label{cor:approxOS}
    Let $F$ and $\mathcal{X}$ be as in \Cref{ass:mainAssumption} such that $\theta \mapsto F(x,\theta)$ is $L_F$ Lipschitz and $x \mapsto J_\theta F(x,\theta)$ is $L_J$ Lipschitz (in operator norm) for all $x \in \RR^n$. Then,  for all $\theta \in \RR^m$,
    \begin{align}
        \|J^\OS x_k(\theta) - J_\theta \bar{x}(\theta)\|_\op \leq \frac{\rho L_F}{1 - \rho} + L_J \|x_{k-1} - \bar{x}(\theta)\|.
        \label{eq:oneStepJacobianBound}
    \end{align} \label{cor:jacobianBound}
\end{corollary}
\begin{proof}
   The result follows from \Cref{lem:matrixIdentity} with $A = J_x F(\bar{x}(\theta),\theta)$, $B = J_\theta F(\bar{x}(\theta),\theta)$ and $\tilde{B} = J_\theta F(x_{k-1},\theta)$ using the fact that $\|B\|_\op \leq L_F$ and $\|B - \tilde{B}\|_\op \leq L_J \|\bar{x}(\theta) - x_{k-1}\|$.
\end{proof}

\begin{remark}[Comparison with implicit differentiation]
\label{rem:implicitDiff}{\rm
In \cite{blondel2022efficient} a similar bound is described for $J^{\ID}$, roughly under the assumption that $x \to F(x,\theta)$ also has $L_J$ Lipschitz Jacobian, one has
\begin{align}    \label{eq:implicitJacobianBound}
    \|J^{\ID} x_k(\theta) - J_\theta \bar{x}(\theta)\|_\op \leq \frac{L_J L_F}{(1 - \rho)^2}\|x_k - \bar{x}(\theta)\| + \frac{L_J}{1 - \rho} \|x_k - \bar{x}(\theta)\|.
\end{align}
For small $\rho$ and large $k$, the main difference between the two bounds \eqref{eq:oneStepJacobianBound} and \eqref{eq:implicitJacobianBound} lies in their first term which is of the same order whenever $\rho$ and $L_J \|\bar{x} - x_{k-1}\|$ are of the same order.    }
\end{remark}

Corollary \ref{cor:jacobianBound} provides a bound on $\|J^\OS x_k(\theta) - J_\theta \bar{x}(\theta)\|_\op$ which is asymptotically proportional to $\rho$. This means that for fast linearly convergent algorithms, meaning $\rho \ll 1$, one-step differentiation provides a good approximation of the actual derivative.
Besides, given $F$ which satisfies \Cref{ass:mainAssumption}, with a given $\rho<1$, not specially small, one can set $\tilde{F}_\theta = F_\theta^K$ for some $K \in \NN$. In this  case, $\tilde{F}$ satisfies assumption \ref{ass:mainAssumption} with $\tilde{\rho} = \rho^K$ and the one-step estimator in \eqref{eq:diffOS} applied to $\tilde{F}$ becomes a $K$-steps estimator on $F$ itself, in other words, we only differentiate through the $K$ last steps of the algorithm.

\begin{example}[Gradient descent]{\rm    Let $f \colon \RR^{n} \times \RR^m \to \RR$ be such that $f(\cdot, \theta)$ is $\mu$-strongly convex ($\mu > 0$) with $L$ Lipschitz gradient for all $\theta \in \RR^m$, then the gradient mapping $F \colon (x,\theta) \mapsto x - \alpha \nabla_x f(x,\theta)$ satisfies \Cref{ass:mainAssumption} with $\rho = \max\{1 - \alpha\mu, \alpha L - 1\}$ which is smaller than $1$ as long as $0 < \alpha < 2 /L$. The optimal step size $\alpha = 2 /( L+\mu)$ leads to a contraction factor $\rho = 1-2\mu / (L + \mu)$. Assuming that $\nabla_{x\theta}^2 f$ is also $L$ Lipschitz (in operator norm), \Cref{cor:approxOS} holds with $L_F = L_J = 2 L / (\mu + L) \leq 2$. For step size $1/L$, \Cref{cor:jacobianBound} holds with $L_F = L_J \leq 1$ and $\rho = 1 - \mu/L$. In both cases, the contraction factor $\rho$ is close to $0$ when $L / \mu \simeq 1$, that is the problem is well conditioned. As outlined above, we may consider $\tilde{F}_\theta = F_\theta^K$ in which case \Cref{cor:jacobianBound} applies with a smaller value of $\rho$.}
\end{example}

\subsection{Superlinar  and quadratic algorithms}
\label{sec:estimateAlgorithms}
The one-step Jacobian estimator in \eqref{eq:diffOS} as implemented in Algorithm \ref{alg:oneStep} is obviously not an exact estimator in the sense that one does not necessarily have $J^{\OS} x_k(\theta) \to J_\theta \bar{x}(\theta)$ as $k \to \infty$. However, it is easy to see that this estimator is exact in the case of exact single-step algorithms, meaning $F$ satisfies $F(x,\theta) = \bar{x}(\theta)$ for all $x,\theta$. Indeed, in this case, one has $J_x F(x,\theta) = 0$ and $J_\theta F(x,\theta) = J_\theta \bar{x}(\theta)$ for all $x,\theta$. Such a situation occurs, for example, when applying Newton's method to an unconstrained quadratic problem. This is a very degenerate situation as it does not really make sense to talk about an ``iterative algorithm'' in this case. It turns out that this property of being ``asymptotically correct'' remains valid for very fast algorithms, that is, algorithms that require few iterations to converge, the archetypical example being Newton's method for which we obtain quantitative estimates.

\subsubsection{Super-linear algorithms}
The following is a typical property of fast converging algorithms.
\begin{assumption}[Vanishing Jacobian]{\rm 
     Assume that $F \colon \RR^{n} \times \RR^m \to \RR^n$ is $C^1$ and that the recursion $x_{k+1} = F_\theta(x_k)$ converges globally, locally uniformly in $\theta$ to the unique fixed point $\bar{x}(\theta)$ of $F_\theta$ such that  $J_x F(\bar{x}(\theta),\theta) = 0$.
     \label{ass:superLinear}}
\end{assumption}
Note that under \Cref{ass:superLinear}, it is always possible to find a small neighborhood of $\bar{x}$ such that $\|J_x F_\theta\|_{\op}$ remains small, that is, \Cref{ass:mainAssumption} holds locally and \Cref{lem:fixedPoint} applies. Furthermore, it is possible to show that the derivative estimate is asymptotically correct as follows. 
\begin{corollary}[Jacobian convergence]\label{cor:superlinearCvgt}
    \label{cor:superLinear}
    Let $F \colon \RR^{n} \times \RR^m \to \RR^n$ be as in \Cref{ass:superLinear}. Then $J^{\OS} x_k(\theta) \to J_\theta \bar{x}(\theta)$ as $k \to \infty$, and $J^{\OS} \bar{x}(\theta) = J_\theta \bar{x}(\theta)$.
\end{corollary}
\begin{proof}
    Since $J_x F(\bar{x}(\theta),\theta) = 0$, implicit differentiation of the fixed point equation reduces to $J_\theta \bar{x}(\theta) = J_\theta F(\bar{x}(\theta), \theta)$, and the result follows by continuity of the derivatives.
\end{proof}

\begin{example}[Superlinearly convergent algorithm]{\rm 
    Assume that $F$ is $C^1$ and for each $\rho > 0$, there is $R>0$ such that $\|F_\theta(x) - \bar{x}(\theta)\| \leq \rho \|x - \bar{x}(\theta)\|$ for all $x,\theta$ such that $\|x - \bar{x}(\theta)\| \leq R$. Then \Cref{cor:superLinear} applies as for any $v$
    \begin{align*}
        J_x F(\bar{x}(\theta), \theta) v = \lim_{t \to 0}  \frac{F(\bar{x}(\theta)+tv, \theta) - \bar{x}(\theta)}{t}  = 0.
    \end{align*}}
\end{example}

\subsubsection{Quadratically convergent algorithms}

\begin{wrapfigure}{r}{0.3\textwidth}
  \vspace{-30pt}
      \includegraphics[width=.3\textwidth]{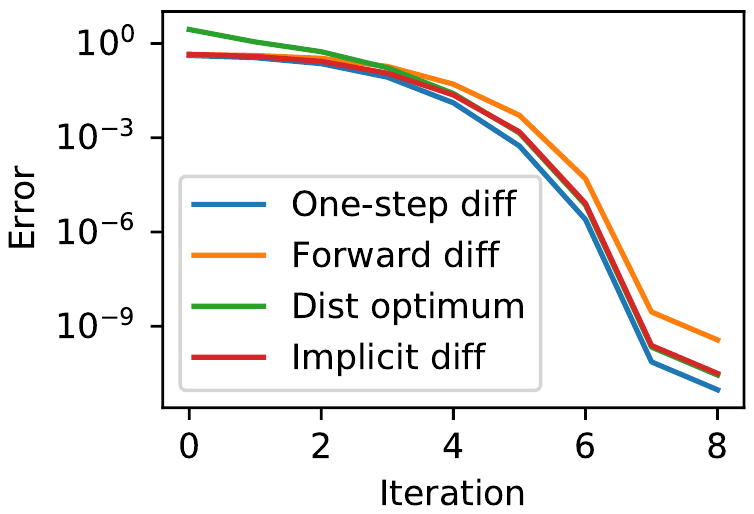}
      \caption{Newton's method quadratic convergence.}
      \label{fig:illustrQuadraticConvergence}
      \vskip-.3in
\end{wrapfigure}

Under additional quantitative assumptions, it is possible to obtain more precise convergence estimates similar to those obtained for implicit differentiation, see \Cref{rem:implicitDiff}.
\begin{corollary}\label{cor:quadCvgt}
    Let $F$ be as in \Cref{ass:superLinear} such that $x \mapsto J_{(x,\theta)} F(x,\theta)$ (joint jacobian in $(x,\theta)$) is $L_J$ Lipschitz (in operator norm). Then, the recursion is asymptotically quadratically convergent and for each $k \geq 1$,
    \begin{align}
        \|J^\OS x_k(\theta) - J_\theta \bar{x}(\theta)\|_\op \leq L_J \|x_{k-1}(\theta) - \bar{x}(\theta)\| .
        \label{eq:oneStepJacobianBoundQuadConv}
    \end{align}
    \label{cor:jacobianBoundQuadConv}
\end{corollary}
\begin{proof}
Following the same argument as in the proof of \Cref{cor:superLinear}, we have
    \begin{align}
        \|J^\OS x_k(\theta) - J_\theta \bar{x}(\theta)\|_\op & = \|J_\theta F(x_k(\theta),\theta) - J_\theta F(\bar{x}(\theta),\theta)\|_\op \leq L_J \|x_{k-1}(\theta) - \bar{x}(\theta)\| .
        \label{eq:oneStepJacobianBoundQuadConvExplicit}
    \end{align}
As for the quadratic convergence, we may assume that $\bar{x}(\theta) = 0$ and  drop the $\theta$ variable to simplify notations. We have $F(0) = 0$ and for all $x$,
   $$F(x) = \int_0^1 J_xF(tx) x dt \leq \|x\| \int_0^1 \|J_xF(tx)\|_{\op}  dt \leq \|x\|^2L_J \int_0^1 t  dt = \frac{L_J\|x\|^2}{2}.
    $$
Thus $L_J/2\|x_{k+1}\|\leq [L_J/2 \|x_k\|]^2$, and asymptotic quadratic convergence follows.\end{proof}

\begin{example}[Newton's algorithm]{\rm
    Assume that $f \colon \RR^n \times \RR^m \to \RR$ is $C^3$ with Lipschitz derivatives, and for each $\theta$, $x \mapsto f(x,\theta)$ is $\mu$-strongly convex. Then Newton's algorithm with backtracking line search satisfies the hypotheses of \Cref{cor:jacobianBoundQuadConv}, see \cite[Sec. 9.5.3]{boyd2004book}. Indeed, it takes unit steps after a finite number of iterations, denoting by $\bar{x}(\theta)$ the unique solution to $\nabla_x f(x,\theta) = 0$, for all $x$ locally around $\bar{x}(\theta)$
    \begin{align*}
        F(x,\theta) = x - \nabla^2_{xx} f(x,\theta)^{-1} \nabla_x f(x,\theta).
    \end{align*}
    We have, $\nabla^2_{xx} f(x,\theta)F(x,\theta) = \nabla^2_{xx} f(x,\theta)x - \nabla_x f(x,\theta)$, differentiating using tensor notations,
    \begin{align*}
        \nabla_{xx}^2 f(x,\theta)J_x F(x,\theta) = \nabla^3 f(x,\theta)[x,\cdot,\cdot] - \nabla^3 f(x,\theta)[F(x, \theta),\cdot,\cdot]
    \end{align*}
    so that $J_x F(\bar{x}(\theta),\theta) = 0$ and  Lipschitz continuity of the derivatives of $f$ implies Lipschitz continuity of $J_x F(x,\theta)$ using the fact that $\nabla_{xx}^2 f(x,\theta) \succeq \mu I$ and that matrix inversion is smooth and Lipschitz for such matrices.
    The quadratic convergence of the three derivative estimates for Newton's method is illustrated on a logistic regression example in \Cref{fig:illustrQuadraticConvergence}.}
\end{example}

\section{Hypergradient descent for bilevel problems}
\label{sec:bilevelOptimization}

Consider the following bilevel optimization problem 
\begin{align*}
    \min_\theta \quad g(x(\theta))
    \quad \mathrm{s.t.}\quad x(\theta) \in {\arg\min}_y f(y,\theta),
\end{align*}
where $g$ and $f$ are $C^1$ functions. We will consider bilevel problems such that the inner minimum is uniquely attained and can be described as a fixed point equation $x = F(x,\theta)$ where $F$ is as in \Cref{ass:mainAssumption}. The problem may then be rewritten as
\begin{align}
    \min_\theta \quad g(x(\theta)) 
    \quad\mathrm{s.t.}\quad x(\theta) = F(x(\theta),\theta), \label{eq:biLevel}
\end{align}
see illustrations in \Cref{ex:contractionsOpti}. Gradient descent (or hyper-gradient) on \eqref{eq:biLevel} using our one-step estimator in \eqref{eq:diffOS} consists in the following recursion
\begin{align}
    \theta_{i+1}  = \theta_i - \alpha J^\OS x_k(\theta_i)^T \nabla g(x_k),
    \label{eq:hyperGradientOS}
\end{align}
where $\alpha > 0$ is a step size parameter. Note that the quantity $J^\OS x_k(\theta)^T \nabla g(x_k)$ is exactly what is obtained by applying backpropagation to the composition of $g$ and Algorithm \ref{alg:oneStep}, without any further custom variation on backpropagation. This section is dedicated to the theoretical guaranties which can be obtained using such a procedure, proofs are postponed to \Cref{app:proofBilevel} .

\subsection{Complexity analysis of different hypergradient strategies}
We essentially follow the complexity considerations in \cite[Section 4.6]{griewank2008evaluating}.
Let $C_F$ denote the computation time cost of evaluating the fixed-point map $F$ and $\omega >0$ be the multiplicative overhead of gradient evaluation, in typical applications, $\omega \leq 5$ (cheap gradient principle~\cite{baur1983complexity}). The time cost of evaluating the Jacobian of $F$ is $n \omega C_F$ ($n$ gradients). Forward algorithm evaluation (i.e., $x_k$) has computational time cost $k C_F$ with a fixed memory complexity $n$.
Vanilla piggyback recursion~\eqref{eq:prodjac} requires $k-1$ full Jacobians and matrix multiplications of costs $n^2 m$.
The reverse-mode of AD has time complexity $k \omega C_F$ (cheap gradient principle on $F_\theta^k$) and requires to store $k$ vectors of size $n$.
Implicit differentiation requires \emph{one} full Jacobian ($\omega C_Fn$) and solution of \emph{one} linear system of size $n \times n$, that is roughly $n^3$.
Finally, one-step differentiation is given by only differentiating a single step of the algorithm at cost $ \omega C_F$.
For each estimate, distance to the solution will result in derivative errors. In addition, automatic differentiation based estimates may suffer from the burn-in effect \cite{scieur2022curse} while one-step differentiation will suffer from a lack of contractivity as in \Cref{cor:jacobianBound}. We summarize the discussion in Table~\ref{tab:complexity}. Let us remark that, if $C_F \geq n^2$, then reverse AD has a computational advantage if $k \leq n$, which makes sense for fast converging algorithms, but in this case, one-step differentiation has a small error and a computational advantage compared to reverse AD.

\begin{remark}[Implicit differentiation: but on which equation?]{\rm \Cref{tab:complexity}  is informative yet formal. In practical scenarios, implicit differentiation should be performed using the simplest equation available, not necessarily $F=0$. This can significantly affect the computational time required. For instance, when using Newton's method $F=-[\nabla^2f]^{-1}\nabla f$, implicit differentiation should be applied to the gradient mapping $\nabla f=0$, not $F$. In typical application $C_{\nabla f} = O(n^2)$, and the dominant cost of implicit differentiation is $O(n^3)$, which is of the same order as the one-step differentiation as $C_F= O(n^3)$ (a linear system needs to be inverted). However, if the implicit step was performed on $F$ instead of $\nabla f$, it would incur a prohibitive cost of $O(n^4)$. In conclusion, the implicit differentiation phase is not only custom in terms of the implementation, but also in the very choice of the equation. }
\end{remark}

\begin{table}[t]
     
    \centering
    \begin{tabular}{lccc}
        \toprule
        Method & Time & Memory & Error sources \\ \midrule
        Piggyback recursion & $kn (\omega C_F + n m)$ & $n(n+m)$ & suboptimality + burn-in \\ 
        AD reverse-mode & $k \omega C_F$ & $k n$ & suboptimality + burn-in \\
        Implicit differentiation & $\omega C_Fn + n^3$ & $n$ &  suboptimality \\ 
        One-step differentiation & $\omega C_F$ & $n$ & suboptimality + lack of contractivity\\ \midrule
        Forward algorithm & $k C_F $ & $n$ & suboptimality \\
        \bottomrule
    \end{tabular}  
    \caption{Time and memory complexities of the estimators in \Cref{sec:threeAlgorithms} (up to multiplicative constant). $F$ has time complexity denoted by $C_F$ and we consider $k$ iterations in $\RR^n$ with $m$ parameter. $\omega$ denotes the multiplicative overhead of evaluating a gradient (cheap gradient principle).\label{tab:complexity} }
    \vskip -.2in
    
\end{table}
\subsection{Approximation analysis of one step differentiation}
The following corollary provides a bound on gradients for problem \eqref{eq:biLevel} obtained by the one-step differentiation as in Algorithm \ref{alg:oneStep}. The bound depends on $\rho$, the contraction factor, and distance to the solution for the inner algorithm.
 The proof is given in \Cref{app:proofBilevel}.
\begin{corollary}
    Let $F \colon \RR^n \times \RR^m \to \RR^n$ be as in \Cref{cor:jacobianBound} and consider the bilevel problem \eqref{eq:biLevel}, where $g$ is a $C^1$, $l_g$ Lipschitz function with $l_\nabla$ Lipschitz gradient. Then, 
    \begin{align*}
        \left\| \nabla_\theta (g \circ \bar{x})(\theta) - J^\OS x_k(\theta)^T \nabla g(x_k) \right\| \leq \frac{\rho L_F l_g}{1 - \rho} + L_Jl_g \|x_{k-1} - \bar{x}(\theta)\| + L_F l_\nabla \|\bar{x}(\theta) - x_k\| .
    \end{align*}
    \label{cor:bilevel}
\end{corollary}
\subsection{Approximate critical points}
The following lemma is known, but we provide a proof for completeness in \Cref{app:proofBilevel}.
\begin{lemma}
    Assume that $f \colon \RR^n \to \RR$ is $C^1$ with $L$ Lipschitz gradient and lower bounded by $f^*$. Assume that for some $\epsilon > 0$, for all $n \in \NN$, 
    $
        \left\|x_{k+1} - x_k + \frac{1}{L} \nabla f(x_k)\right\| \leq \epsilon.
    $
    Then, for all $ K \in \NN$, $K \geq 1$, we have
    $$
        \min_{k = 0, \ldots, K} \|\nabla f(x_k)\|^2 \leq \epsilon^2 +  \frac{2(f(x_0) - f^*)}{LK}.
    $$
    \label{lem:approximateCriticalPoints}
\end{lemma}
The combination of \Cref{cor:bilevel} and \Cref{lem:approximateCriticalPoints} applies to the gradient descent algorithm in \eqref{eq:hyperGradientOS} for problem \eqref{eq:biLevel}, and justifies the fact that it finds approximate critical points. The level of approximate criticality depends on $\rho$ and distance to the solution of the inner problem. For large values of $k$ (many steps on the inner problem), approximate criticality is essentially proportional to $\rho$. For superlinear algorithm, as in \Cref{sec:estimateAlgorithms}, approximate criticality is essentially proportional to the distance to the solution of the inner problem, as for implicit differentiation, see \Cref{rem:implicitDiff}.

\section{Numerical experiments}
We illustrate our findings on three different problems. First, we consider Newton's method applied to regularized logistic regression, as well as interior point solver for quadratic problems. These are two fast converging algorithms for which the results of \Cref{sec:estimateAlgorithms} can be applied and the one-step procedure provides accurate estimations of the derivative with a computational overhead negligible with respect to solution evaluation,  as for implicit differentiation. We then consider the gradient descent algorithm applied to a ridge regression problem to illustrate the behavior of the one step procedure in the context of linearly convergent algorithms.

\begin{figure}
    \centering
    \includegraphics[width = \textwidth]{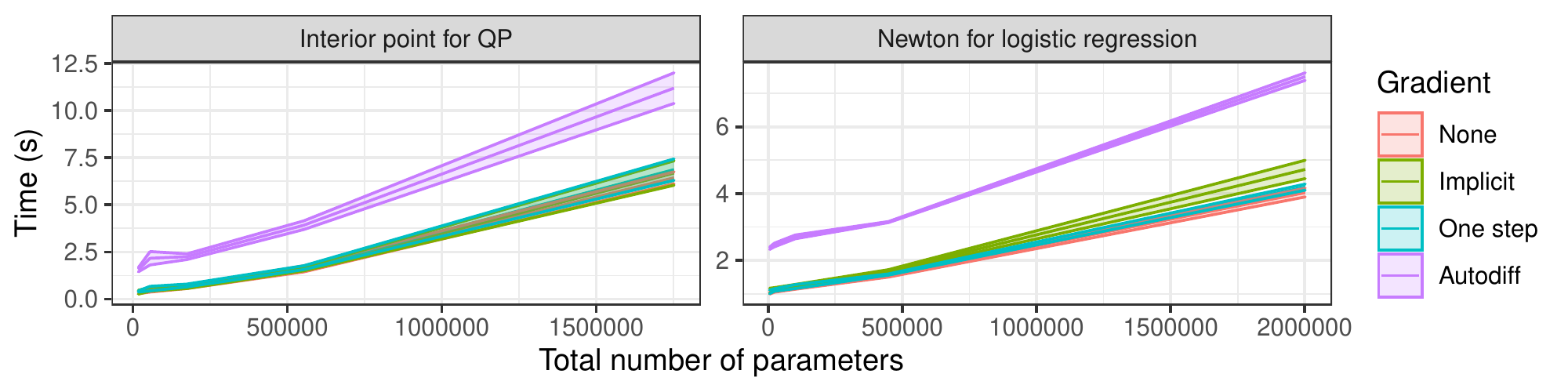}
    \vskip-.1in
    \caption{Left: timing experiment for differentiable quadratic programs. Right timing experiment for differentiation of Newton's algorithm for logistic regression. For Newton's experiment, the one-step estimator coincides with implicit differentiation estimator up to $10^{-12}$ error, and for the interior point experiment, it coincides with the implicit differentiation estimator up to $10^{-6}$ error. Label ``None'' represent solving time and ``Autodiff'', ``Implicit'' and ``One step'' represent solving time and additional evaluation of gradient using each algorithm in \Cref{sec:threeAlgorithms}}
    \label{fig:NewtonIP}
\end{figure}

\paragraph{Logistic regression using Newton's algorithm.}
Let $A \in \RR^{N \times n}$ be a design matrix, the first column being  made of $1$s to model an intercept. Rows are denoted by $(a_1,\ldots, a_N)$. Let $x \in \RR^n$ be a regression parameter and $y \in \{-1,1\}^{N}$. We consider the regularized logistic regression problem
\begin{align}
    \min_{x \in \RR^n} \sum_{i=1}^N \theta_i \ell(\left\langle a_i,x \right\rangle y_i) + \lambda \|x_{-1}\|^2, \label{eq:logisticReg}
\end{align}
where $\ell$ is the logistic loss, $\ell \colon t \mapsto \log(1 + \exp(-t))$, $\lambda > 0$ is a regularization parameter, and $x_{-1}$ denotes the vector made of entries of $x$ except the first coordinate (we do not penalize intercept). This problem can be solved using Newton's method which we implement in \texttt{jax} using backtracking line search (Wolfe condition). Gradient and Hessian are evaluated using \texttt{jax} automatic differentiation, and the matrix inversion operations are performed with an explicit call to a linear system solver. 

We denote by $x(\theta)$ the solution to problem \eqref{eq:logisticReg} and try to evaluate the gradient of $\theta \mapsto \|x(\theta)\|^2/2$ using the three algorithms presented in \Cref{sec:threeAlgorithms}. We simulate data with Gaussian class conditional distributions for different values of $N$ and $n$. The results are presented in \Cref{fig:NewtonIP} where we represent the time required by algorithms as a function of the number of parameters required to specify problem \eqref{eq:logisticReg}, in our case size of $A$ and size of $y$, which is $(n+1)N$.

\Cref{fig:NewtonIP} illustrates that both one-step and implicit differentiation enjoy a marginal computational overhead, contrary to algorithmic differentiation. In this experiment, the one-step estimator actually has a slight advantage in terms of computation time compared to implicit differentiation.

\paragraph{Interior point solver for quadratic programming:} The content of this section is motivated by elements described in \cite{amos2017optnet}, which is associated with a \texttt{pytorch} library implementing a standard interior point solver. Consider the following quadratic program (QP):
\begin{align}
    \min_{x \in \RR^n} \quad \frac{1}{2} x^T Q x + q^Tx 
  \quad \mathrm{s.t.} \quad A x = \theta,\quad Gx \leq h, \label{eq:QP}
\end{align}
where $Q \in \RR^{n \times n}$ is positive definite, $A \in \RR^{m \times n}$ and $G \in \RR^{p \times n}$ are matrices, $q \in \RR^n$, $\theta \in \RR^m$ and $h\in \RR^p$ are vectors. We consider $x(\theta)$ the solution of problem \eqref{eq:QP} as a function of $\theta$, the right-hand side of the equality constraint. We implemented in \texttt{jax} a standard primal-dual Interior Point solver for problem \eqref{eq:QP}. Following \cite{amos2017optnet}, we use the implementation described in \cite{vandenberghe2010cvxopt}, and we solve linear systems with explicit calls to a dedicated solver. For generic inputs, this algorithm converges very fast, which we observed empirically. Differentiable programming capacities of \texttt{jax} can readily be used to implement the automatic differentiation and one-step derivative estimators without requiring custom interfaces as in \cite{amos2017optnet}. Indeed, implicit differentiation for problem \eqref{eq:QP} was proposed in \cite{amos2017optnet} with an efficient \texttt{pytorch} implementation. We implemented these elements in \texttt{jax} in order to evaluate $J_\theta x(\theta)$ using implicit differentiation. 

We consider evaluating the gradient of the function $\theta \mapsto \|x(\theta)\|^2 / 2$ using the three algorithms proposed in \Cref{sec:threeAlgorithms}. We generate random instances of QP in \eqref{eq:QP} of various sizes. The number of parameters needed to describe each instance is $n(n+1) + (n+1)m + (n+1)p$.  The results are presented in \Cref{fig:NewtonIP} where we represent the time required by algorithms as a function of the number of parameters required to specify problem \eqref{eq:QP}. In all our experiments, the implicit and one-step estimates agree up to order $10^{-6}$.
From \Cref{fig:NewtonIP}, we see that both one-step and implicit differentiation enjoy a marginal additional computational overhead, contrary to algorithmic differentiation. 

\begin{figure}[t]
    \centering
    \includegraphics[width=1\textwidth]{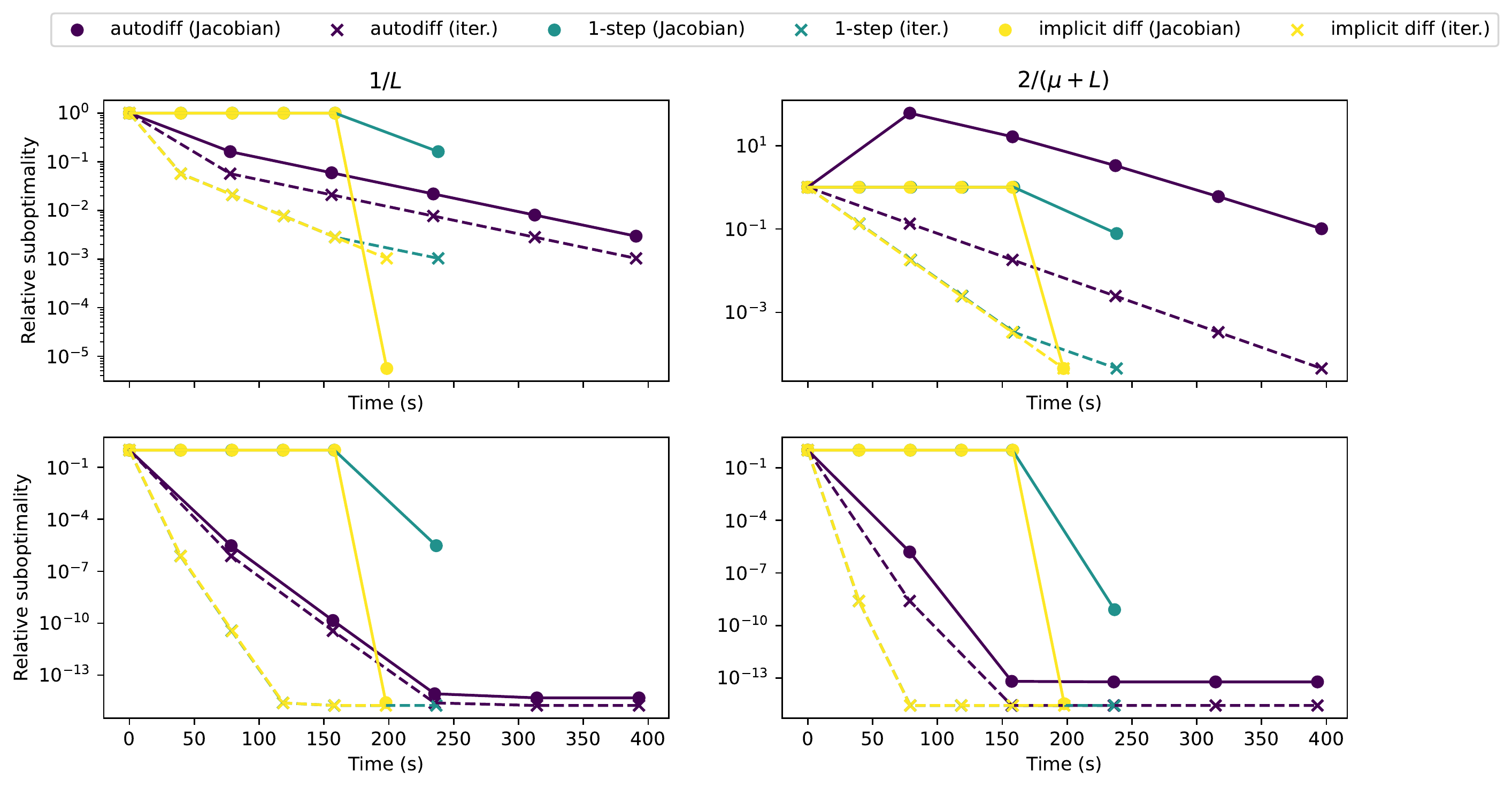}
    \caption{Differentiation of gradient descent for solving weighted Ridge regression on \texttt{cpusmall}. Top line: effective condition number of 1000. Bottom line: effective condition number of 100. Left column: small learning rate $\frac{1}{L}$. Right column: big learning rate $\frac{2}{\mu+L}$. Dotted lines represent the lack of optimality of the iterates whereas filled lines represent the lack of optimality of the Jacobians.}
    \label{fig:gd}
\end{figure}
\paragraph{Weighted ridge using gradient descent.}
We consider a weighted ridge problem with $A \in \RR^{N \times n}$, $y \in \RR^N$, $\lambda > 0$ and a vector of weights $\theta \in \RR^N$:
\[
    \bar x(\theta) = \arg\min_{x \in \RR^n} f_\theta(x) = \frac12 \sum_{i=1}^N \theta_i (y_i - \langle a_i, x \rangle)^2 + \frac{\lambda}{2} \| x \|^2 .
\]
We solve this problem using gradient descent with adequate step-size $F(x, \theta) = x - \alpha \nabla f_\theta(x)$ with $ x_0(\theta) = 0 $,
and we consider the $K$-step truncated Jacobian propagation $\tilde F = F_\theta^K$ with $K=1/\kappa$ where $\kappa$ is the effective condition number of the Hessian.
Figure~\ref{fig:gd} benchmarks the use of automatic differentiation, one-step differentiation, and implicit differentiation on the dataset \texttt{cpusmall} provided by LibSVM~\cite{chang2011libsvm}.
We monitor both the lack of optimality $\| x_k(\theta) - \bar x(\theta) \|^2$ of the iterates, and the lack of optimality $\| J_\theta x_k(\theta) - J_\theta \bar x(\theta) \|^2$ of the Jacobian matrices.
As expected, implicit differentiation is faster and more precise but requires {\em custom implicit system to be implemented} (or the use of an external library).
One-step differentiation is significantly faster than full unrolling and in most cases, provides adequate numerical precision. We reproduce the burn-in phenomenon described in \cite{scieur2022curse}.

\section{Conclusion}
We studied the one-step differentiation, also known as Jacobian-free backpropagation, of a generic iterative algorithm, and provided convergence guarantees depending on the initial rate of the algorithm. In particular, we show that one-step differentiation of a quadratically convergent algorithm, such as Newton's method, leads to a quadratic estimation of the Jacobian. A future direction of research would be to understand how to extend our findings to the nonsmooth world as in~\cite{bolte2022automatic} for linearly convergent algorithms.

\bibliographystyle{plain}
\bibliography{references}

\newpage

\appendix

\section{Proof of \Cref{sec:bilevelOptimization}}
\label{app:proofBilevel}

\begin{proof}[of \Cref{cor:bilevel}]
    We have $\nabla_\theta (g \circ \bar{x})(\theta) = J_\theta \bar{x}(\theta)^T \nabla g(\bar{x}(\theta))$, and
    \begin{align*}
        J_\theta \bar{x}(\theta)^T \nabla g(\bar{x}(\theta))-J^{\OS} x_k(\theta)^T \nabla g(x_k) =& (J_\theta \bar{x}(\theta)^T -J^{\OS} x_k(\theta)^T)  \nabla g(\bar{x}(\theta)) \\
        & -J^{\OS} x_k(\theta)^T (\nabla g(x_k) - \nabla g(\bar{x}(\theta)).
    \end{align*}
    The result follows from the triangular inequality combined with \Cref{cor:jacobianBound} and the following 
    \begin{align*}
        \|(J_\theta \bar{x}(\theta)^T - J^{\OS} x_k(\theta)^T)  \nabla g(\bar{x}(\theta))\| &\leq \|J_\theta \bar{x}(\theta) -  J^{\OS} x_k(\theta)\|_\op \|\nabla g(\bar{x}(\theta)) \| \\
        &\leq l_g\|J_\theta \bar{x}(\theta) -  J^{\OS} x_k(\theta)\|_\op ,\\
    \end{align*}
    and
    \begin{align*}
        \|J^{\OS} x_k(\theta)^T (\nabla g(x_k) - \nabla g(\bar{x}(\theta))\| & \leq \|J^{\OS} x_k(\theta)\|_\op \| \nabla g(x_k) - \nabla g(\bar{x}(\theta))\| \leq L_F l_\nabla \|\bar{x}(\theta) - x_k\|
    \end{align*}
\end{proof}

\begin{proof}[of \Cref{lem:approximateCriticalPoints}]
    We have, using the ``descent lemma" (see \cite{beck2017first}), for all $k \in 0 \ldots K$,
    \begin{align*}
        f(x_{k+1}) -  f(x_k)  &\leq \left\langle \nabla f(x_k), x_{k+1} - x_k \right\rangle + \frac{L}{2} \|x_{k+1} -x_k\|^2 \\
        & = \frac{L}{2} \left( 2\left \langle \frac{\nabla f(x_k)}{L}, x_{k+1} - x_k \right\rangle +  \|x_{k+1} -x_k\|^2 \right)\\
        & = \frac{L}{2} \left(\left\|x_{k+1} -x_k + \frac{\nabla f(x_k)}{L}\right\|^2 - \|\nabla f(x_k)\|^2 \right)\\
         & \leq \frac{L}{2} \left(\epsilon^2 -  \min_{k = 0, \ldots, K} \|\nabla f(x_k)\|^2\right).
    \end{align*}

    Summing for $k = 0,\ldots, K$, we have
    \begin{align*}
        f^* - f(x_0) \leq  f(x_{K+1}) - f(x_0) \leq \frac{KL}{2} \left(\epsilon^2 -  \min_{k = 0, \ldots, K} \|\nabla f(x_k)\|^2\right).
    \end{align*}
    and we deduce, by using concavity of the square root, that 
    \begin{align*}
        \min_{k = 0, \ldots, K} \|\nabla f(x_k)\|^2 \leq \epsilon^2 + \frac{2(f(x_0) - f^*)}{LK}.
    \end{align*}
\end{proof}

% \section{Code to reproduce experiments}
% We provide codes to reproduce our numerical experiments in a separate zip archive. The experiments are implemented in \texttt{python} using \texttt{jax} library. We also use \texttt{matplotlib} and \texttt{ggplot2} library in \texttt{R} to generate graphics. The ridge regression experiment is based on \texttt{cpusmall} dataset provided by LibSVM.
\end{document}